\documentclass[a4paper,11pt,reqno]{amsart}

\usepackage[utf8]{inputenc}
\usepackage[T1]{fontenc}

\usepackage{paralist}				
\usepackage{letterspace}
\usepackage{ifdraft}
\usepackage{cite}
\usepackage{microtype}
\usepackage{tikz,tikz-cd}
\usepackage{flafter}
\usepackage{hyperref}

\usepackage{amsmath}
\usepackage{amssymb}
\usepackage{amsthm}				
\usepackage{mathtools}				
\usepackage{mathrsfs}				
\usepackage{dsfont}				

\newcommand{\mf}{\mathfrak}
\newcommand{\mb}{\mathbb}
\newcommand{\mc}{\mathcal}
\newcommand{\msc}{\mathscr}

\newcommand{\IN}{\mathbb{N}}
\newcommand{\IZ}{\mathbb{Z}}
\newcommand{\IR}{\mathbb{R}}

\newcommand{\ga}{\alpha}
\newcommand{\gb}{\beta}
\newcommand{\gc}{\gamma}
\newcommand{\gC}{\Gamma}

\newcommand{\gep}{\varepsilon}

\newcommand{\gp}{\varphi}

\newcommand{\bP}{\mathbf{P}}
\newcommand{\bE}{\mathbf{E}}
\newcommand{\Var}{\mathbf{Var}}

\newcommand{\ssq}{\subseteq}

\newcommand{\Poi}{\mathsf{Poi}}
\newcommand{\Exp}{\mathsf{Exp}}
\newcommand{\Ber}{\mathsf{Ber}}
\newcommand{\dto}{\stackrel{\mathscr{D}}{\longrightarrow}}
\newcommand{\Di}{\,\textnormal{d}}
\newcommand{\Beta}{\operatorname{B}}

\newcommand{\dist}{\operatorname{dist}}
\newcommand{\clos}{\operatorname{clos}}
\newcommand{\rt}{\textnormal{root}}

\newcommand{\Cut}{\mathsf{Cut}}
\newcommand{\perc}{\operatorname{perc}}

\newcommand{\cater}{\textnormal{CP}}
\newcommand{\Star}{\textnormal{ST}}
\newcommand{\binary}{\textnormal{CBT}}

\newtheorem{lemma}{Lemma}
\newtheorem{cor}[lemma]{Corollary}
\newtheorem{prop}[lemma]{Proposition}
\newtheorem{thm}[lemma]{Theorem}
\theoremstyle{remark}
\newtheorem{rem}[lemma]{Remark}
\newtheorem{ex}[lemma]{Example}
\theoremstyle{definition}
\newtheorem{defn}[lemma]{Definition}

\title{A Modification of the Random Cutting Model}
\author{Fabian Burghart}
\address{Department of Mathematics, Uppsala University, S-751 06 Uppsala, Sweden}
\email{fabian.burghart@math.uu.se}


\begin{document}

\begin{abstract}
  We propose a modification to the random destruction of graphs: 
  Given a finite network with a distinguished set of sources and targets, remove (cut) vertices at random, discarding components that do not contain a source node. We investigate the number of cuts required until all targets are removed, and the size of the remaining graph. This model interpolates between the random cutting model going back to Meir and Moon \cite{MeirMoon1970} and site percolation. We prove several general results, including that the size of the remaining graph is a tight family of random variables for compatible sequences of expander-type graphs, and determine limiting distributions for binary caterpillar trees and complete binary trees. 
\end{abstract}

\maketitle

\section{Introduction}

Investigating the behaviour of trees when randomly removing vertices was first done by Meir and Moon in \cite{MeirMoon1970}. The process in question starts with a rooted tree, where at every step a uniformly chosen vertex [or edge, but for the purpose of this paper we will focus on vertex-removals] is deleted, and all remaining components that do not include the root vertex are discarded. Since the process naturally stops once the root node has been cut [or isolated], the question of interest is the random number of cuts needed to reach this state. In \cite{MeirMoon1970}, the expected value and the variance of this random variable for a random labelled tree is found. 

Panholzer \cite{panholzer2003non,panholzer2006cutting} used generating functions to determine the asymptotic distribution of the cutting number for non-crossing trees and for certain families of simply generated trees. Janson \cite{janson2006random} generalised these results, obtaining the asymptotic distribution for conditioned Galton-Watson trees, by approximating the cutting number with a sum of independent random variables (see also the alternative proof in \cite{addario2014cutting}). Using a similar strategy, a limit law for complete binary trees is proven in \cite{janson2004random}. This work was later extended by Holmgren to binary search trees \cite{holmgren2010random} and split trees \cite{holmgren2011weakly}.

The random cutting model relates to both the record problem (as was observed in \cite{janson2006random}) and to fragmentation processes, where the genealogy of the procedure gives rise to the so-called Cut-tree which has become both a useful tool in obtaining results on the cutting number as well as an object of independent interest. See e.g. \cite{bertoin2012fires}, \cite{broutin2017cutting}, or \cite{berzunza2017cut}.

In recent years, modifications of the original cutting model have been discussed: Kuba and Panholzer regarded the case of isolating a leaf, a general node, or multiple nodes simultaneously (instead of isolating the root) in \cite{kuba2008isolating2,kuba2008isolating,kuba2880multiple}, and Cai, Holmgren et al. proposed and investigated the $k$-cut model in \cite{cai2019k,cai2019cutting,berzunza2019k}, where a node is only removed after it has been cut for the $k$-th time.

In this paper, a different modification of the cutting model is introduced, where, additionally to one or several root vertices (which we will call sources) a second set of vertices (targets) are given. This allows for defining a stopping time for the cutting procedure on the graph by looking at the first moment when all of the targets have been removed (i.e. the sources have been \emph{separated} from the targets). We can then ask several natural questions, about the number of cuts necessary to separate the two sets, about which vertex that was the last to be removed, and about the size of the remaining graph. The aim of the paper is to investigate how separation interpolates, between the cutting model and site percolation, how these questions relate to each other, and 
what general properties hold for this modification of the cutting model, and finally to give some examples.

\subsection*{Plan of the paper} 
In Section 2, we will fix some notation, define our model in both a continuous and discrete time setting, and present a first example.

Section 3 will contain several basic estimates, and we will formalise the imprecise notion that separation interpolates between the cutting model and site percolation (Propositions~\ref{prop:cutting} and \ref{prop:perc}). Since the latter is commonly defined on a locally finite, infinite graph, this requires the right definition that enables us to approximate such a graph by finite graphs, all while respecting the choice of sources and targets, Definition~\ref{defn:exhaust}. 

Section 4 begins by determining the probability that a fixed subgraph occurs at some point in the cutting procedure, which is then used to obtain the probability for this subgraph to be the remaining part at separation. This leads to Theorem~\ref{thm:tight}, which could be regarded as the main result of the paper and gives sufficient conditions for the size of the graph at separation to be a tight sequence of random variables when the graph approximates a locally finite infinite graph in the sense of Definition~\ref{defn:exhaust}. 

In Section 5, our scope will focus on rooted trees, since their recursive nature can be used to simplify many of the arguments and calculations. Here, we will investigate a relevant polynomial that already came up (in slightly different form) in the works of Devroye et al., \cite{devroye2011note, 2013transversals}.

Section 6 will contain three examples of rooted trees: Binary caterpillars, star-shaped graphs, and complete binary trees. These examples are selected to illustrate the techniques (and limitations thereof) of the earlier Sections. Finally, in Section 7 we will give a brief compilation of further research questions.

\section{Cutting procedures}

\subsection*{Some notation}
We will always use $G=(V(G),E(G))$ to denote a graph, consisting of its vertex and edge set, but will shorten the notation to $V=V(G)$ and $E=E(G)$ if there is no ambiguity from the context. If the graph in question is a tree, we prefer the symbol $\mc T$ over $G$. Further notation for special graphs will be introduced as required.
Since most subgraphs we will consider are induced and therefore uniquely determined by their vertex set, we will not distinguish between an induced subgraph and its vertex (sub-)set.

If two vertices $v,w$ are neighbours, we also use the notation $v\sim w$. More generally, we write $\dist(v,w)$ for the graph distance between vertices $v,w$. In the case where $A,B\ssq V(G)$ are subsets, $\dist(A,B)$ is to be understood as $\min\{\dist(v,w):v\in A,w\in B\}$.

Given any set $A\ssq V(G)$ and a fixed set $S$ of source nodes, we define the closure of $A$ to be 
\[
 \clos_S(A):=A\cup S \cup \{v\in V(G):v\sim w \text{ for some }w\in A\}.
\]
The (exterior) boundary of $A$ is defined as $\partial_S A:=\clos_S(A)\setminus A$. In other words, the vertices in $\partial_S A$ are precisely the vertices not in $A$ that are in $S$ or neighbour some vertex in $A$. Note that this implies e.g. $\partial_S \emptyset := S$. (Attributing this special role to $S$ in this definition turns out to be useful throughout the paper).

We use the symbols $\gC$ and $\Beta$ to denote the Gamma and Beta function, respectively, and will make use of the identities 
\[
 \Beta(x,y):=\int_0^1 u^{x-1}(1-u)^{y-1}\,\Di u
 =\frac{\gC(x)\gC(y)}{\gC(x+y)}
\]
and
\[
 \Beta(n+1,m+1)=\frac{1}{n+m+1}\binom{n+m}{n}^{-1}=\frac{1}{n+1}\binom{n+m+1}{n+1}^{-1}
\]
for $x,y>0$ and $n,m\in\IN$ (where, in the last expression, the roles of $n$ and $m$ can be reversed due to the symmetry of the Beta function), cf. Chapter 6 in \cite{handbook}.

\subsection*{Cutting and separation} Consider a finite simple connected graph $G=(V,E)$ with a distinguished subset $S\ssq V$ whose vertices are referred to as sources. Now, proceed as follows:
\begin{enumerate}[1.]
 \item Choose a vertex $v\in V$ uniformly at random, and remove it -- together with all edges incident to $v$ -- from the graph. This will potentially split the graph into connected components, in which case we only keep the components containing sources. 
 \item Iterate step 1, where the randomness in choosing the node is assumed to be independent from everything that happened previously. 
 \item The process terminates once the graph contains no more vertices. Equivalently, this happens as soon as the last source node has been removed. 
\end{enumerate}
This defines a finite random sequence of induced subgraphs 
\begin{equation}\label{eq:cp1}
 G=: G_0 \supsetneq G_1 \supsetneq ... \supsetneq G_{r-1} \supsetneq G_r = \emptyset,
\end{equation}
where we denoted the empty subgraph consisting of no vertices with $\emptyset$ by a slight abuse of notation. Observe that therefore, we can view the cutting procedure as a time-discrete stochastic process on the state space consisting of subgraphs of $G$, and as such it is in fact a Markov-chain. We will denote this process by $\Cut(G)$.

Instead of describing the cutting procedure by the sequence of graphs we obtain, as in \eqref{eq:cp1}, we can equivalently keep track of the removed vertices, $v_1,...,v_r$. We will use both styles of bookkeeping interchangeably, depending on which one is more suitable for the task at hand. 

Introducing a second set of distinguished vertices, $T$, whose vertices we refer to as targets, we can now consider the following functionals of the cutting process:
\begin{itemize}
 \item The cutting number $\mf C(G)$. This is merely the number of cuts until the last source node is cut, or equivalently, until the remaining graph is empty, i.e. $\mf C(G)=\inf\{i\geq 0: G_i=\emptyset\}$. Note that this does not rely on $T$.
 \item The separation number $\mf S(G)$, defined to be the number of cuts until the remaining graph does not contain target nodes anymore (independently of how many sources are still present). In other words, $\mf S(G)=\inf\{i\geq 0:V(G_i)\cap T=\emptyset\}$. We say that at this time, separation (of $S$ and $T$) occurs. 
 \item The separation subgraph $G_{\mf S}:= G_{\mf S(G)}$, defined to be the random subgraph of $G$ at separation. 
 \item The separation node $v_{\mf S}$, denoting the last node that was removed before separation occurred. 
\end{itemize}

\begin{figure}[h]
 \includegraphics[]{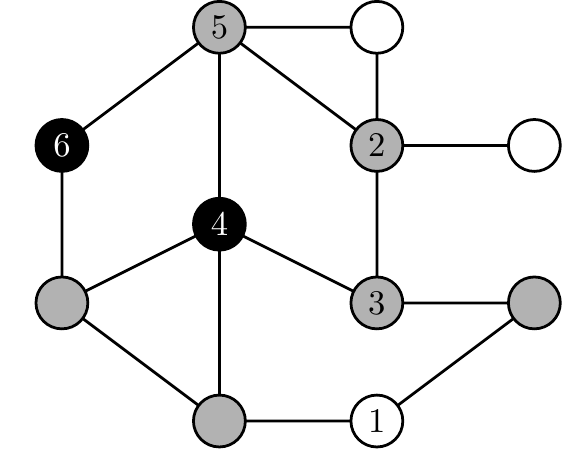}
 \caption{A graph $G$, with source nodes in black and target nodes in white. The number $i$ at a vertex indicates that this vertex was chosen by the $i$-th cut. Note that cuts 2 and 3 each remove an additional vertex. In this example, $\mf S(G)=5$, $\mf C(G)=6$ and $G_{\mf S}$ is a path on three vertices, beginning with the source labelled 6.}
\end{figure}

\subsection*{The continuous-time model} As has been observed previously by \cite{janson2006random} and since been brought to effective use, the above cutting model is equivalent to a model where each node is equipped with a random alarm clock whose alarm triggers after time $X_v$, $v\in V$. Whenever an alarm rings, and the corresponding node is still in the graph at this time, that node will be removed together with any new components that do not contain source nodes. Here, to ensure equivalence to the discrete-time cutting model above, we assume that $(X_v)_{v\in V}$ is an i.i.d. family of $\Exp(1)$-distributed random variables -- while any i.i.d. family of continuous random variables would suffice, the memoryless property will be useful later. 

This once again yields a monotone stochastic process of subgraphs of $G$, but now parametrised by continuous time, $(G^c_t)_{t\in[0,\infty)}$. We will denote this process by $\Cut^c(G)$. However, $G_t$ will only attain finitely many different graphs, and we will still denote those graphs by $G_0,G_1,G_2,...$ in order of occurrence, as before. Hence, we can denote by $G_{t^-}$ the graph that was attained by $\Cut^c(G)$ immediately before time $t$; so, $G_{t^-}=G_t$ iff no cut happened at time $t$. 

Note that there are two ways of generalising the random variables $\mf C$ and $\mf S$ to the continuous-time setting: By default, $\mf C$ and $\mf S$ respectively denote the quantities 
\[
  \mf C = \inf\{i\in\IN : G_i=\emptyset\} \text{ \ and \ } 
  \mf S = \inf\{i\in\IN : V(G_i)\cap T=\emptyset\},
\]
exactly as before, while $\mf C_c$ and $\mf S_c$ denote
\[
 \mf C_c = \inf\{t\geq0 : G_t=\emptyset\} \text{ \ and \ }
 \mf S_c = \inf\{t\geq0 : V(G_t)\cap T=\emptyset\},
\]
respectively.

\subsection*{A first example}
Denote by $P_n$ for $n\geq0$ a path on $2n+1$ vertices. Let the middle vertex be the unique source of $P_n$ and let $T$ be the two leaves. 

\begin{prop}\label{prop:path}
 Let $X\sim \Poi(\ln 2)$, and let $\mb P$ be the distribution of $1+X$ conditioned on $X$ being at least $1$. Then $\mf S[P_n]\dto \mb P$. In other words, 
 \[
   \bP[\mf S[P_n]\leq k] \to \bP[X\leq k-1|X\geq 1]=\sum_{j=1}^{k-1} \frac{(\ln 2)^j}{j!}
 \]
 for all $k\geq 0$. 
\end{prop}
\begin{proof}
 Embed $P_n, n\geq 1$ in $[0,1]$ by mapping the $2n+1$ vertices to equidistant points $p_{n,i}=\frac i {2n}$ for $i=0,1,...,2n$ such that neighbouring vertices have a distance of exactly $\frac 1{2n}$. Note that independently of $n$, the root vertex will always be at $1/2$, and the leaves at $0$ and $1$. 
 
 Given that a cut is chosen uniformly among the discrete vertices of the component containing the source, the limit of the cut position as $n$ tends to infinity will be uniformly drawn from an interval containing $1/2$ under the above embedding, which follows from Theorem 7.8 in \cite{billingsley1968convergence}. Thus, in the limit, we can investigate a space-continuous cutting model: We begin with the interval $[0,1]$, we then choose a point from the interval uniformly at random, and consider the component containing $1/2$. Now, we repeat this procedure, where the randomness in every step is only dependent on the current interval. In analogy to the discrete cutting number, we can define $\mf S[0,1]$ to be the time of the first cut where the remaining interval is of the form $(a,b)$ with $0<a<b<1$. By this construction, we have $\mf S[P_n]\dto \mf S[0,1]$.
 
 In this continuous setting, $\mf S[0,1]=k$ for $k\geq 2$ means that the first $k-1$ cuts all happened on the same side of $1/2$, i.e. in $(1/2,1]$ or in $[0,1/2)$, and the $k$-th cut took place on the other side. Since the entire setup is symmetric w.r.t. $1/2$, assume without loss of generality that the first cut happens inside $[1/2,1]$. The probability for $\{\mf S[0,1]=k\}$ to happen is thus expressed by the following iterated integral, where $c_1,c_2,...$ denote the first, second, ... cut: 
 \begin{equation}\label{eq:path}
  \bP[\mf S[0,1]=k] 
  = 2\int_{1/2}^1 \int_{1/2}^{c_1} \cdots \int_{1/2}^{c_{k-2}} \frac{1/2}{c_{k-1}} \cdot\frac{\Di c_{k-1}}{c_{k-2}} \cdots \frac{\Di c_2}{c_1} \Di c_1 .
 \end{equation}
 Here, the innermost integrand $1/(2c_{k-1})$ is the probability of $c_k$ falling into the interval $[0,1/2]$ given that the last cut was $c_{k-1}$; equivalently, given that the interval for $c_k$ to choose from was $[0,c_{k-1})$. The nested integrals in \eqref{eq:path} then arise from repeatedly applying the law of total probability and can be evaluated iteratively to yield $\frac{\ln(2)^{k-1}}{(k-1)!}$, as required. 
\end{proof}

This result is in stark contrast with $\mf C(P_n)$, for which it is known that
\[
 \frac{\mf C(P_n)-2\ln n}{\sqrt{2\ln n}} \dto N(0,1)
\]
as $n\to\infty$ (cf e.g. Example 8.1 in \cite{janson2006random} and references therein).

\section{Monotonicity, cutting, and site percolation}

The purpose of this section is to state and prove some basic properties concerning the behaviour of the separation time when changing the underlying graph or the selection of source/target nodes. For this purpose, we will write $\mf S(G;S,T)$ instead of $\mf S(G)$ whenever there is potential ambiguity over the choice of source and target nodes. The following definition will prove useful:

\begin{defn}
 Let $G=(V,E)$ be a finite graph, and let $S,T\ssq V$ be distinguished sets of sources and targets. A vertex $v\in V$ is called \emph{essential} if $G$ contains a path $sv_1...v_kt$ containing $v$, such that $s$ and $t$ are the unique vertices on this path that are contained in $S$ and $T$, respectively. Call a vertex \emph{non-essential} if no such path exists. 
\end{defn}

Observe that it is possible both for source and target nodes to be non-essential -- this is the case precisely if every relevant path passes through another source or target node.

If two graphs $G,G'$ on a common vertex set with respective source and target vertices $S,T,S',T'\ssq V$ are given, and if they are such that any sequence of removed vertices $v_1,\dots,v_r$ that separates $S$ from $T$ in $G$ also separates $S'$ from $T'$ in $G'$, then $\mf S(G';S',T')\leq \mf S(G;S,T)$ and $\mf S_c(G';S,T)\leq \mf S_c(G;S,T)$. This argument proves each claim in the next proposition: 

\begin{prop}\label{prop:mono}
 Let $G=(V,E)$ be a finite connected graph, and let $S,T\ssq V$ be the sets of source and target nodes, respectively. The following is true:
 \begin{enumerate}[(i)]
  \item If $G'$ is a subgraph of $G$ endowed with source and target nodes inherited from $G$, then $\mf S(G')\leq \mf S(G)$. 
  \item If $S$ and $T$ are replaced by subsets $S'\ssq S$ and $T'\ssq T$ respectively, then $\mf S(G;S',T')\leq \mf S(G;S,T)$. 
  \item If $T$ is replaced by a set $T'$ of essential vertices such that every path from $S$ to $T$ intersects $T'$, then $\mf S(G;S,T)\leq \mf S(G;S,T')$. 
  \item If $S$ is replaced by a set $S'$ of essential vertices such that every path from $S$ to $T$ intersects $S'$, then $\mf S(G;S,T)\leq \mf S(G;S',T)$.
 \end{enumerate}
 All three claims remain correct if $\mf S$ is replaced by its continuous-time counterpart $\mf S_c$. 
\end{prop}

\begin{rem}
 It follows from (i) in Proposition~\ref{prop:mono} that one has the following chain of inequalities for any connected graph $G=(V,E)$ with $S,T\ssq V$: 
 \begin{equation}\label{eq:bounds}
  \mf S(\mc T) \leq \mf S(G) \leq \mf S(K_V)
 \end{equation}
 where $\mc T$ is a spanning tree of $G$ and $K_V$ is the complete graph on the vertex set $V$. 
 Observe that for some graphs, both the upper and lower estimates may hold with equality: Indeed, assume $K_V$ has a single source note $s$, and set $\mc T$ to be the spanning tree of $K_V$ that is a star graph centered on $s$. Running the cutting procedure on either $K_V$ or $\mc T$ will amount to removing one vertex at a time until the last source vertex is removed, thus showing equality.
 
 It should also be noted that the continuous-time separation time is a lot more stable under changes to $G$, as demonstrated by the following lemma:
\end{rem}

\begin{lemma}
 Let $G=(V,E)$ and $S,T\ssq V$ as above. Denote by $G'$ the graph obtained by removing all non-essential vertices from $G$. Then $\mf S_c[G]=\mf S_c[G']$ deterministically. 
\end{lemma}

\begin{proof}
 Denote by $G=G_0,G_1,G_2,...$ and $G'=G'_0,G'_1,G'_2,...$ the continuous-time cutting procedures on $G$ and $G'$, respectively. As $V(G')\ssq V(G)$, we can use the same randomness for $X_v$ in both graphs when $v\in V(G')$. Hence we certainly have $\mf S_c[G]\geq \mf S_c[G']$. Now assume that at time $t\geq\mf S_c[G']$, there still is a path $sv_1...v_kt$ connecting $S$ and $T$ in $G_t$ (thus necessarily using non-essential nodes). This path must contain vertices $v_i,v_j$ with $i\leq j$ s.t. $v_i\in S$, $v_j\in T$ and none of the vertices on the path between $v_i$ and $v_j$ are in either $S$ or $T$. However, by definition all vertices $v_i,v_{i+1},...,v_j$ are essential and thus in $G'$, contradicting the assumption that separation in $G'$ already occurred.
\end{proof}

The following proposition asserts that the additional freedom of choosing target nodes for the separation number can be used to obtain the original cutting number. In other words, $\mf S(G)$ can be understood as a generalisation of $\mf C(G)$. 

\begin{prop}\label{prop:cutting}
 Let $G=(V,E)$ be a finite connected graph, and let $S,T\ssq V$ be the sets of source and target nodes, respectively. Then, we have $\mf S(G)\leq \mf C(G)$ deterministically, with equality if $S\ssq T$. Conversely, if $\mf S(G)=\mf C(G)$ in distribution, then $S\ssq T$. Moreover, all of those statements also hold true for $\mf S_c(G)$ and $\mf C_c(G)$ in the continuous-time model. 
\end{prop}
\begin{proof}
 At time $\mf C(G)$, the remaining graph is empty, so separation must have occurred already. Thus $\mf S(G)\leq \mf C(G)$.
 
 If $S\ssq T$, then separation will occur as soon as the last source node has been removed, at which time the remaining graph will be empty. Thus $\mf S(G)=\mf C(G)$.
 Conversely, if there exists $v\in S\setminus T$ then $G_{\mf S}$ contains $v$ with some positive probability $p_0$. If this happens, $\mf S(G)\leq \mf C(G)-1$, so $\bE[\mf C(G)]-\bE[\mf S(G)]\geq p_0$, and equality in distribution cannot hold. 
 
 For the continuous-time model, only the last argument requires modification: Once again, if $v\in S\setminus T$ then $G_{\mf S}$ contains $v$ with probability $p_0>0$. In this case, $\mf S_c(G) < X_v \leq \mf C_c(G)$, and we have 
 \[
  \bE[\mf C_c(G)]-\bE[\mf S_c(G)] \geq p_0 \bE[X_v-\mf S_c(G)\mid v\in V(G_{\mf S})]=p_0
 \]
 since $X_v\sim\Exp(1)$ is memoryless.
\end{proof}

\subsection*{Separation and site percolation}

In this section, we show that in a certain sense, the continuous-time separation model on an infinite graph $G$ with infinite distance $\dist(S,T)$ contains the site percolation model on $G$. 

More precisely, recall that for Bernoulli site percolation in an infinite graph $G$, every node is independently kept with some probability $p\in[0,1]$ and otherwise rejected, thus giving a random subgraph of $G$. We denote by $\perc_S(p)$ the probability that the $\Ber(p)$-site percolation on $G$ exhibits an infinite cluster containing at least one vertex of $S$. 

\begin{defn}\label{defn:exhaust}
 Let $G$ be a locally finite, infinite connected graph, containing two subsets $S,T\ssq V(G)$. We say that the sequence $(G^{(n)})_{n\geq 1}$ of finite induced subgraphs of $G$ \emph{exhausts} $G$ if the following conditions are satisfied:
 \begin{enumerate}[(i)]
  \item The $G^{(n)}$ are connected subgraphs satisfying
  \[
   G^{(1)}\ssq G^{(2)} \ssq ... \ssq G \quad \text{ and } \quad 
   \bigcup_{n\geq1} V\left(G^{(n)}\right) = V(G).
  \]
  \item The set $S$ is entirely contained in $G^{(n)}$ for all $n$ (and understood to be the set of source nodes of $G^{(n)}$), and each subgraph $G^{(n)}$ is endowed with the target set $T^{(n)} := \left(T\cap V\left(G^{(n)}\right)\right)\cup \partial_\emptyset\left(G\setminus G^{(n)}\right)$.
 \end{enumerate}
\end{defn}

Observe that the target nodes $T^{(n)}$ are indeed a subset of $V\left(G^{(n)}\right)$ and will be non-empty even if $T=\emptyset$. Moreover, condition (ii) necessitates that $S$ is finite. 

\begin{prop}\label{prop:perc}
 Let $G$ be a locally finite, infinite graph with a finite set of source nodes, $S$, and let $T=\emptyset$. Assume that $(G^{(n)})_{n\geq 1}$ exhausts $G$. 
 
 Then, 
 \begin{equation}\label{eq:perc}
  \lim_{n\to\infty} \bP\left[\mf S_c\big(G^{(n)}\big)\geq x\right] = \perc_S(e^{-x}).  
 \end{equation}
\end{prop}

\begin{proof}
 Note first that if $T=\emptyset$, then (using the notation of Definition \ref{defn:exhaust}) $\dist\left(S,T^{(n)}\right)\to\infty$ as $n\to\infty$. Indeed, there would otherwise be a bound $C\in\IN$ with $\dist\left(S,T^{(n)}\right) < C$. Consider the neighbourhood $B_C(S):=\{v\in V(G):\dist(v,S)\leq C\}$ of $S$. Since $B_C(S)$ is finite, it will eventually be contained in all $G^{(n)}$, contradicting the notion that $T^{(n)}=\partial_\emptyset\left(G\setminus G^{(n)}\right)$ contains vertices in $B_{C-1}(S)$. 

 Recall next that independently removing each vertex $v$ of $G$ at a random time $X_v\sim\Exp(1)$ gives rise to the monotonous coupling of Bernoulli site percolation for all parameters $p\in[0,1]$ (cf. \cite[p.138]{lyons2017probability}). Indeed, at time $x\in[0,\infty]$, the graph we observe that way is a sample of $\Ber(e^{-x})$-site percolation on $G$. We can couple the process obtained in this way to the continuous-time cutting model by restricting our attention to the intersection of $G^{(n)}$ with those percolation clusters that intersect $S$. 
 
 To show $\geq$ in \eqref{eq:perc}, assume that $\Ber(e^{-x})$-site percolation exhibits an infinite cluster which intersects $S$, such cluster necessarily intersects $T^{(n)}$ as well and hence, for each $n$, contains a path connecting $S$ with $T^{(n)}$. By the coupling indicated above, this path must then also be present in the sample of the continuous-time cutting model on $G^{(n)}$ at time $x$. Therefore, $\perc_S(e^{-x}) \leq \bP\left[\mf S_c\big(G^{(n)}\big)\geq x\right]$, and letting $n$ tend to $\infty$ yields $\liminf_{n\to\infty} \bP\left[\mf S_c\big(G^{(n)}\big)\geq x\right]\geq \perc_S(e^{-x})$.
 
 For the other inequality, suppose now that $\Ber(e^{-x})$-site percolation does not exhibit an infinite cluster intersecting $S$, so that the total mass of clusters intersecting $S$ is bounded by some finite integer, say $k$. By the second assumption, we have $\dist(S,T^{(n)}) >k$ for all but finitely many $n$. However, this implies that eventually, the clusters intersecting $S$ cannot intersect $T^{(n)}$, which, for the coupled cutting procedure, means that separation in $G^{(n)}$ must have occurred before time $x$. So,
 \[
  \bP\left[ |V(\text{Clusters intersecting } S)|\leq k \right] \leq 
  \bP\left[\mf S_c\big(G^{(n)}\big)< x\right].
 \]
 Taking the limit superior for $n\to\infty$ yields 
 \[
  \limsup_{n\to \infty} \bP\left[\mf S_c\big(G^{(n)}\big)\geq x\right] \leq 1-\bP\left[ |V(\text{Clusters intersecting } S)|\leq k \right],
 \]
 which implies the existence of the limit and $\leq$ in \eqref{eq:perc} after passing to the limit for $k\to\infty$ as well.
\end{proof}

\section{Visiting probability of subgraphs and size of the separation graph}

Consider a finite simple connected graph $G$ with $S,T\ssq V$ as usual. The aim of this section is to determine the probability that at some time $i\geq 1$, the cutting procedure $\Cut(G)$ will produce a specific subgraph $G_*$.

\begin{lemma} \label{lem:grpdfs}
 Fix an induced subgraph $G_*$ of $G$ with every component of $G_*$ containing at least one source node. Then, for all times $t\geq 0$ in the continuous-time cutting model, we have
 \begin{equation}\label{eq:pdfgraph}
  \bP\left[G_t \supseteq G_*\right]= e^{-|G_*|t} 
 \end{equation}
 and
 \begin{equation}\label{eq:pdfgraph2}
  \bP\left[G_t=G_* \right] = e^{-|G_*|t}\left(1-e^{-t}\right)^{|\partial_S G_*|}.
 \end{equation}
 Moreover, consider $v_*\in\partial_S G_*$. Then
 \begin{equation}\label{eq:pdfgraph3}
  \bP\left[G_{X_{v_*}} = G_* \mid X_{v_*} \right] = e^{-|G_*|X_{v_*}} \left(1-e^{-X_{v_*}}\right)^{|\partial_S G_*|-1}.
 \end{equation}
\end{lemma}
\begin{proof}
 The random subgraph $G_t$ contains $G_*$ if and only if at time $t$, none of the clocks $X_v$ for $v\in V(G_*)$ have rung yet and every component is still attached to a source node. Hence, as the $X_v$ are independent $\Exp(1)$-distributed random variables, we obtain
 \begin{align*}
  \bP\left[G_t \supseteq G_* \right]
  =\prod_{v\in V(G_*)} \bP[X_v \geq t]
  =e^{-|G_*|t}. 
 \end{align*}
 Similarly, $G_{t}=G_*$ if and only if at time $t$, none of the clocks for $v\in V(G_*)$, but all of the clocks for $v\in \partial_S G_*$ (this includes the nodes in $S$ which are not in $G_*$, by definition of $\partial_S$ in section 2) have rung, which yields \eqref{eq:pdfgraph2}. 
 
 For the final claim, note that at time $X_{v_*}$, all nodes in $G_*$ must still be intact, whereas all nodes in $\partial_S G_*\setminus\{v_*\}$ must have already been cut. Now proceed as in the case for a deterministic time.
\end{proof}

Of course, equations \eqref{eq:pdfgraph} and \eqref{eq:pdfgraph2} in the previous proposition are equivalent. One can easily obtain \eqref{eq:pdfgraph} from \eqref{eq:pdfgraph2} by summing over all graphs that contain $G_*$, and for the other direction, applying the inclusion-exclusion principle suffices.

\begin{cor}\label{cor:vispr}
 Fix an induced subgraph $G_*$ of $G$ with every component of $G_*$ containing at least one source node. Let $v_*\in\partial_S G_*$. Denote the $i$-th graph obtained in the cutting process by $G_i$ and the $i$-th cut node by $v_i$. Then
 \begin{equation}\label{eq:gvlast}
  \bP\left[\exists i\geq 1:G_i=G_*\text{ and }v_i=v_*\right]
  =\frac{1}{|\partial_S G_*|} \binom{|G_*|+|\partial_S G_*|}{|\partial_S G_*|}^{-1}
 \end{equation}
 and therefore
 \begin{equation}\label{eq:gstar}
  \bP\left[\exists i\geq 1:G_i=G_*\right] = \binom{|G_*|+|\partial_S G_*|}{|\partial_S G_*|}^{-1}.
 \end{equation}
\end{cor}

\begin{proof}
 We prove the result by considering the continuous-time model instead. There, the event $B(v_*):=\{\exists i\geq1:G_i=G_*\text{ and }v_i=v_*\}$ translates to the event $\{G_{X_{v_*}}=G_*\}$. Hence, we obtain the first result from integrating over $X_{v_*}\sim \Exp(1)$ in \eqref{eq:pdfgraph3} by substituting $e^{-x}=u$:
 \begin{align*}
 \bP[B(v_*)]
  &= \int_0^\infty \left(e^{-x}\right)^{|G_*|}\left(1-e^{-x}\right)^{|\partial_S G_*|-1} \cdot e^{-x} \Di x \\
  &= \int_0^1 u^{|G_*|} \left(1-u\right)^{|\partial_S G_*|-1} \Di u\\
  &=\Beta(|G_*|+1,|\partial_S G_*|) \\
  &=\frac{1}{|\partial_S G_*|}\binom{|G_*|+|\partial_S G_*|}{|\partial_S G_*|}^{-1},
 \end{align*}
 as required. Finally, the observation that
 \[
  \{\exists i\geq 1:G_i=G_*\}=\biguplus_{v_*\in\partial G_*} B(v_*)
 \]
 lets us obtain \eqref{eq:gstar} from \eqref{eq:gvlast}.
\end{proof}

\begin{defn}\label{defn:admissable}
 Let $G$ be a finite connected graph equipped with sources $S$ and targets $T\neq\emptyset$. An induced subgraph $G_*$ is called \emph{admissible} if the probability $\bP[G_{\mf S}=G_*]$ is positive. We denote by $\msc A_m(G)$ the set of all admissible subgraphs $G_*$ of $G$ of size $|V(G_*)|=m$. 
\end{defn}

\begin{lemma}\label{lem:admissable}
 Assume $T\neq \emptyset$. An induced subgraph $G_*\ssq G$ is admissible iff $G_*$ contains no target nodes and every component of $G_*$ contains at least one source node.
\end{lemma}
\begin{proof}
 That the stated conditions for $G_*$ are necessary for $G_*$ being admissible is evident from the definitions of the cutting procedure and the separation time. To show that they are also sufficient, label the vertices in $\partial_S G_*$ by $v_1,...,v_r$ in such a way that after removing $v_1,...,v_{r-1}$ from $G$, the remaining graph still contains a path connecting $S$ and $T$ (such a boundary vertex exists since there must be at least one such path in $G$, and it cannot be contained entirely in $G_*$ by assumption). Then, one way of realising the event $\{G_{\mf S}=G_*\}$ is by removing the vertices $v_1,...,v_r$ in this order. Hence, assuming $|V(G)|=n$:
 \[
  \bP[G_{\mf S}=G_*] \geq \prod_{i=1}^{r} \frac{1}{n+1-i} > 0 
 \]
 thus concluding the proof.
\end{proof}

Using similar methods as in the proofs to Lemma \ref{lem:grpdfs} and Corollary \ref{cor:vispr}, we can establish the following connection between the graph $G_{\mf S}$ at separation and the continuous-time separation number $\mf S_c$:

\begin{prop}\label{prop:sepgr}
 Let $G_*$ be an admissible subgraph of $G$. Then,
 \begin{equation}\label{eq:pdfsepgr}
  \bP[G_{\mf S}=G_*] = \sum_{v_*\in \partial_S G_*}
  \int_0^1 u^{|G_*|-1}(1-u)^{|\partial_S G_*|-1}  \bP\left[\mf S_c(H[v_*]) \geq -\ln u \right]\Di u,
 \end{equation}
 where $H[v_*]$ denotes the graph obtained from $G$ in the following way: Remove all vertices in $G_*$ and in $\partial_S G_* \setminus\{v_*\}$ from $G$, and from what remains, let $H[v_*]$ be the connected component containing $v_*$. We endow $H[v_*]$ with source node $v_*$ and target nodes $T\cap V(H[v_*])$.
\end{prop}
\begin{proof}
 Fix a vertex $v_*\in\partial_S G_*$, and assume that this is the last node to be removed for separation to occur.
 We observe first that by definition of the separation number, any graphs obtained by $\Cut(G)$ before separation must have contained a path from $v_*$ to $T$. In particular, the last graph before separation occurred contained such a path, which additionally was not passing through any other nodes in $G_*$ or $\partial_S G_*$ and must have therefore been contained in $H[v_*]$. The existence of such a path means, however, that the graph $H[v_*]$ is not yet separated. Thus, by transitioning from the discrete to the continuous-time model, we obtain
 \begin{align}\label{eq:condsepgr}
  &\bP[G_{\mf S}=G_*, v_{\mf S}=v_* \mid X_{v_*}]\notag\\
  &\qquad= \bP\left[G_{X_{v_*}}=G_* \text{ and $v_*$ connects to $T$ in $G_{X_{v_*}^-}$}\mid X_{v_*}\right]\notag\\
  &\qquad= \bP\left[G_{X_{v_*}}=G_*\mid X_{v_*}\right] \bP\left[\mf S_c(H[v_*])\geq X_{v_*}\mid X_{v_*} \right],
 \end{align}
 where conditional independence holds true because 
 \[
  \{G_{X_{v_*}}=G_*\}\ \text{ and }\ \{\mf S_c(H[v_*])\geq X_{v_*}\}
 \] 
 are events on vertex sets which only share $v_*$. Conditioned on $X_{v_*}$ being $x$, the event $\{\mf S_c(H[v_*])\geq X_{v_*}\}$ amounts to the existence of a path from $v_*$ to the set of target nodes in $H[v_*]$, none of whose clocks have rung yet at time $X_{v_*}$. On the other hand, without the conditioning, the same event is equivalent to the existence of a path from a neighbour of $v_*$ to the set of target nodes in $H[v_*]$. Hence,
 \[
  \bP\left[\mf S_c(H[v_*])\geq X_{v_*}\mid X_{v_*}=x \right] = \frac{\bP\left[\mf S_c(H[v_*])\geq x \right]}{e^{-x}}.
 \]
 In light of \eqref{eq:pdfgraph3} from Lemma \ref{lem:grpdfs}, we can now rewrite equation \eqref{eq:condsepgr} as
 \begin{multline}\label{eq:sepgrpr}
  \bP[G_{\mf S}=G_*,v_{\mf S}=v_*\mid X_{v_*}=x]\\
   = e^{-|G_*|x} \left(1-e^{-x}\right)^{|\partial_S G_*|-1} \frac{\bP\left[\mf S_c(H[v_*])\geq x \right]}{e^{-x}}.
 \end{multline}
 Finally, observe that, with $\mu_X$ denoting the distribution of $X_{v_*}$,
 \begin{align*}
  \bP[G_{\mf S}=G_*] = \sum_{v_*\in\partial_S G_*} \int_0^\infty \bP[G_{\mf S}=G_*, v_{\mf S}=v_*\mid X_{v_*}=x] \Di \mu_X(x),
 \end{align*}
 so that, after plugging in the expression from \eqref{eq:sepgrpr} and using $X_{v_*}\sim \Exp(1)$, we obtain
 \begin{multline*}
  \bP[G_{\mf S}=G_*]\\
  = \sum_{v_*\in\partial_S G_*} \int_0^\infty (e^{-x})^{|G_*|} 
    \left(1-e^{-x}\right)^{|\partial_S G_*|-1} 
    \frac{\bP\left[\mf S_c(H[v_*])\geq x \right]}{e^{-x}}\cdot e^{-x}\Di x,
 \end{multline*}
 which only differs from \eqref{eq:pdfsepgr} by the substitution $e^{-x}=u$.
\end{proof}

\begin{rem}\label{rem:equality}
 The (unconditioned) probability $\bP[\mf S_c(H[v_*])\geq X_{v_*}]$,  has a number of equivalent versions. Indeed, if we consider any $G$ with $S=\{v_*\}$, then we have the following equalities:
 \begin{align*}
  \bP[\mf S_c(G)\geq X_{v_*}] 
  &= \bP[\mf S_c(G)\geq \mf C_c(G)] 
   = \bP[\mf S(G)\geq \mf C(G)]\\
  &= \bP[G_{\mf S}=\emptyset] 
   = \bP[|G_{\mf S}|=0].
 \end{align*}
 Moreover, in the first three formulations, the strict inequality ``$>$'' is impossible, so one could just as well write ``$=$'' there. 
 
 Additionally, since $\mf S_c(H[v_*])\leq X_{v_*}$, we have the estimate 
 \begin{equation}\label{eq:est1}
  \bP\left[\mf S_c(H[v_*]) \geq -\ln u \right]\leq \bP\left[X_{v_*}\geq -\ln u\right] = u \qquad \forall u\in[0,1].
 \end{equation}
\end{rem}

Proposition~\ref{prop:sepgr} enables us to make the upper bound $\mf S(K_V)$ from \eqref{eq:bounds} more explicit:
\begin{cor}\label{cor:complete}
 Let $K_V$ be a complete graph on a vertex set $V$ of cardinality $|V|=n$ with sources $S$ and targets $T$. Then,
 \begin{equation}\label{eq:complsize}
  \bP\left[\left|K_{V,\mf S}\right|=m\right] = \frac{|T|}{m+1}\binom{n}{m+1}^{-1} \left[\binom{n-|T|}{m}-\binom{n-|S\cup T|}{m}\right],
 \end{equation}
 for $m=1,\dots,n$. Furthermore, for times $t=1,\dots,n-1$,
 \begin{multline}\label{eq:compltime}
  \bP[\mf S(K_V)=t] = \frac{1}{n-t+1}\binom{n}{n-t+1}^{-1}\\ \cdot\left[|T|\binom{n-|T|}{n-t} + |S|\binom{n-|S|}{n-t} - |S\cup T|\binom{n-|S\cup T|}{n-t}\right].
 \end{multline}
\end{cor}
\begin{proof}
 According to Proposition~\ref{prop:sepgr}, we have 
 \begin{multline*}
  \bP[|K_{V,\mf S}|=m]\\ 
  = \sum_{G_*\in \msc A_m(K_V)} \sum_{v_*\in\partial_S G_*} \int_0^1 u^{m-1}(1-u)^{|\partial_S G_*|-1} \bP[\mf S_c(H[v_*])\geq -\ln u]\,\Di u.
 \end{multline*}
 For $m\geq 1$, this can be evaluated: $\partial_S G_*$ always consists of all the vertices not in $G_*$, so $|\partial_S G_*|=n-m$, and since the graphs $H[v_*]$ only contain the vertex $v_*$, we obtain
 \[
  \bP[\mf S_c(H[v_*])\geq -\ln u]=\begin{cases} 0 & \text{ if } v_*\notin T \\ u & \text{ if } v_*\in T \end{cases}.
 \]
 Since at separation all targets must have been removed, this yields
 \begin{align*}
  \bP[|K_{V,\mf S}|=m] &= \sum_{G_*\in\msc A_m(K_V)} |T|\int_0^1 u^m(1-u)^{n-m-1}\,\Di u \\
  &= |\msc A_m(K_V)| \cdot |T| \cdot B(m+1,n-m).
 \end{align*}
 Finally, by Lemma~\ref{lem:admissable}, the graphs in $\msc A_m(K_V)$ consist of $m$ vertices, none of which are targets, unless the $m$ vertices are chosen from $V\setminus(S\cup T)$. Hence 
 \[
  |\msc A_m(K_V)| = \binom{n-|T|}{m}-\binom{n-|S\cup T|}{m}
 \]
 and \eqref{eq:complsize} follows. 
 
 For the second part, we first show that
 \begin{equation}\label{eq:compltime2}
  \bP[\mf S(K_V)=t] = \bP[|K_{V,\mf S}|=n-t] + \frac{g}{n-t+1}\binom{n}{t-1}^{-1}
 \end{equation}
 for some explicit constant $g$.
 Observe that separation at time $t$ must occur after cutting $t$ vertices, either leaving behind some non-empty graph $K_{V,\mf S}$ (in which case $|K_{V,\mf S}|=n-t$, as the vertices are removed one at a time on $K_V$ -- this case is thus covered by the first summand on the right-hand side of \eqref{eq:compltime2}), or leaving behind an empty graph $K_{V,\mf S}$. This second case occurs if and only if $K_{V,t-1}$ contains exactly one source (which will then have to be cut at time $t$) and at least one target. For the sake of this proof, denote the set of these subgraphs by $\msc G_t$. Hence
 \[
  \bP[\mf S(K_V)=t, K_{V,\mf S}=\emptyset] 
  = \sum_{G_*\in \msc G_t} \frac{\bP[K_{V,t-1}=G_*]}{|G_*|}.
 \]
 Observe that all $G_*\in\msc G_t$ have exactly $n-t+1$ vertices, and that if such a graph is obtained in the cutting procedure, it has to happen at time $t-1$. Thus, invoking equation \eqref{eq:gstar} from Corollary~\ref{cor:vispr}, we obtain
 \[
  \bP[K_{V,t-1}=G_*] = \bP[\exists i: K_{V,i}=G_*] = \binom{n}{t-1}^{-1}
 \]
 and therefore 
 \[
  \bP[\mf S(K_V)=t, K_{V,\mf S}=\emptyset] = \frac{|\msc G_t|}{n-t+1}\binom{n}{t-1}^{-1}.
 \]
 This accounts for the second summand in \eqref{eq:compltime2}, with  $g=|\msc G_t|$. To obtain equation \eqref{eq:compltime}, we note that for building an element of $\msc G_t$, one needs to choose one source vertex and $n-t$ non-source vertices, unless neither of these vertices are targets. Hence, 
 \begin{equation}\label{eq:sizeGt}
  |\msc G_t|=|S|\binom{n-|S|}{n-t} - |S\setminus T|\binom{n-|S\cup T|}{n-t}. 
 \end{equation}
 Now, plugging in equations \eqref{eq:complsize} and \eqref{eq:sizeGt} into equation \eqref{eq:compltime2} yields \eqref{eq:compltime} after some minor simplifications. 
\end{proof}

Recall that a family of real-valued random variables $X_i, i\in I$, is \emph{tight} if for all $\gep>0$, there exists a constant $M$ such that $\bP[|X_i|\geq M]<\gep$ for all $i\in I$, cf. \cite[p.37]{billingsley1968convergence}.
We conclude this section by showing the following tightness result about the size of $G_{\mf S}$:

\begin{thm}\label{thm:tight}
 Let $(G^{(n)})_{n\geq1}$ be a sequence of finite induced subgraphs exhausting the locally finite, infinite graph $G$ equipped with subsets $S,T\ssq V(G)$. Define $a_{m,n} = \sum_{A\in \msc A_m(G^{(n)})} |\partial_S A|$ and let $a_m:= \lim_{n\to\infty} a_{m,n}$. Assume that there are constants $b>0$ and $L\geq0$ such that 
 \begin{enumerate}[(i)]
  \item $|\partial_S A|\geq bm+1$ for all $A\in\msc A_m(G^{(n)})$ and all $m\geq L$, and 
  \item the function $f(x)=\sum_{m=L}^\infty a_m x^m$ has radius of convergence at least $\frac{b^b}{(b+1)^{b+1}}$ and satisfies
  \begin{equation}\label{eq:integrable}
   \int_0^1 f(x(1-x)^b)\Di x <\infty
  \end{equation}
 \end{enumerate}
 Then the sizes of the separation graphs, $\big|G^{(n)}_{\mf S}\big|$, form a tight family of random variables.
\end{thm}

Observe that, since $c:=\max_{x\in[0,1]} x(1-x)^b=\frac{b^b}{(b+1)^{(b+1)}}$, condition (ii) requires that the radius of convergence, $r$, of the power series $f$ is at least $c$. In case $r>c$, this condition is trivially satisfied as the integrand is bounded. However, in case of equality $r=c$, $f$ has a singularity at $c$ by virtue of Pringsheim's theorem, cf. \cite[Theorem IV.6]{flajolet2009analytic}, and \eqref{eq:integrable} is a non-trivial requirement. 
We also remark that condition (i) above is a variant of the notion of expander graphs, which play a crucial role in the theory of percolation, see e.g. \cite{ABS04}. 

\begin{proof}
 We first show that for fixed $m$, the sequence $a_{m,n}$ is nondecreasing and eventually constant. Recall from definition \ref{defn:exhaust} that the sets $T^{(n)}$ consist of two parts, namely $T\cap V\left(G^{(n)}\right)$ and vertices that are incident to edges leaving $G^{(n)}$. By copying the respective argument from the proof of Proposition \ref{prop:perc}, we can show again that $\dist\left(S, T^{(n)}\setminus T\right)\to\infty$ as $n\to\infty$. Hence, for $n$ sufficiently large, the closed neighbourhood of $S$ of radius $m$ will become independent of $n$, and so will $a_{m,n}$. Moreover, since the $G^{(n)}$ are monotonously growing, any admissible subgraph in $G^{(n)}$ will also be admissible in $G^{(n+1)}$ and the number of boundary vertices will not decrease when changing from $G^{(n)}$ to $G^{(n+1)}$. Hence $a_{m,n}$ is nondecreasing in $n$. In particular, the limit $a_m$ exists, is finite, and an upper bound to $a_{m,n}$ for all $n$.
 
 Let $M\geq L$. We now apply Proposition \ref{prop:sepgr}, where we set $p_{v_*}(x):= \bP[\mf S_c(H[v_*])\geq -\ln x]$ (note that this still depends on $G_{\mf S}$ as well!) for brevity. Summing over all $m\geq M$ and all $A\in\msc A_m(G^{(n)})$ yields
 \begin{align*}
  \bP\left[\big|G^{(n)}_{\mf S}\big|\geq M\right]
  &=\sum_{m=M}^\infty \sum_{A\in \msc A_m(G^{(n)})} \int_0^1 x^{m-1}(1-x)^{|\partial_S A|-1} \sum_{v_*\in\partial_S A} p_{v_*}(x)\Di x\\
  &\leq \sum_{m=M}^\infty \sum_{A\in \msc A_m(G^{(n)})} |\partial_S A| \int_0^1 x^m(1-x)^{|\partial_S A|-1}\Di x,
 \end{align*}
 where we applied the estimate \eqref{eq:est1}. Using assumption (i), we get
 \begin{align*}
  \bP\left[\big|G^{(n)}_{\mf S}\big|\geq M\right]
  &\leq \sum_{m=M}^\infty \sum_{A\in \msc A_m(G^{(n)})} |\partial_S A| \int_0^1 (x(1-x)^b)^m \Di x\\
  &=\sum_{m=M}^\infty \int_0^1 a_{m,n}(x(1-x)^b)^m\Di x.
 \end{align*}
 Then, by the above argument on the monotonicity of $a_{m,n}$, we obtain
 \begin{equation*}
  \bP\left[\big|G^{(n)}_{\mf S}\big|\geq M\right] \leq \sum_{m=M}^\infty \int_0^1 a_m (x(1-x)^b)^m \Di x. 
 \end{equation*}
 By monotone convergence, we moreover obtain for all $M\geq L$ that 
 \begin{multline*}
  \sum_{m=M}^\infty \int_0^1 a_m(x(1-x)^b)^m\Di x 
  =\int_0^1 \sum_{m=M}^\infty a_m(x(1-x)^b)^m\Di x\\
  \leq \int_0^1 \sum_{m=L}^\infty a_m(x(1-x)^b)^m\Di x
  =\int_0^1 f(x(1-x)^b)\Di x
 \end{multline*}
 which is finite by assumption (ii). Hence the tail of the series on the left-hand side converges to zero, and therefore $\bP\left[\big|G^{(n)}_{\mf S}\big|\geq M\right] \to 0$ uniformly in $n$ as $M\to\infty$ as well. 
\end{proof}

\begin{rem}\label{rem:tight}
 Tightness of the sizes of the separation graphs $\left|G_{\mf S}^{(n)}\right|$ is a helpful property in order to translate limit laws from the cutting times $\mf C\left(G^{(n)}\right)$ to the separation times $\mf S\left(G^{(n)}\right)$: Assume that there exist sequences $\ga_n$ and $\gb_n>0$ such that 
 \[
  \frac{\mf C\left(G^{(n)}\right)-\ga_n}{\gb_n} \dto X \qquad \text{as } n\to\infty
 \]
 for a random variable $X$ with positive variance. If $\gb_n\to\infty$ and $\left|G_{\mf S}^{(n)}\right|$ forms a tight family of random variables, then the deterministic estimate 
 \[
  \mf S\left(G^{(n)}\right) \leq \mf C\left(G^{(n)}\right) \leq \mf S\left(G^{(n)}\right) + \left|G_{\mf S}^{(n)}\right|
 \]
 implies 
 \[
  \left|\frac{\mf C\left(G^{(n)}\right)-\ga_n}{\gb_n}-\frac{\mf S\left(G^{(n)}\right)-\ga_n}{\gb_n}\right| 
  \leq \frac{\left|G_{\mf S}^{(n)}\right|}{\gb_n}\to 0 
 \]
 in probability, and hence 
 \begin{equation*}\label{eq:seplaw}
  \frac{\mf C\left(G^{(n)}\right)-\ga_n}{\gb_n} \dto X \qquad \text{as } n\to\infty.
 \end{equation*}
\end{rem}

\section{Separating trees}

As we saw in Lemma \ref{lem:grpdfs}, the expressions for probabilities of events occurring in $\Cut_c$ happen to be polynomial expressions in $e^{-x}$. This is not surprising: If the event $E$ is only depending on the subgraph at a time $x\geq0$, then 
\[
 \bP[E]=\sum_{G_*\in E} \bP[G_x=G_*]
\]
where the summation ranges over all subgraphs $G_*\ssq G$ in $E$, and from equation \eqref{eq:pdfgraph2}, this is polynomial in $e^{-x}$. As it turns out, this is especially useful for rooted trees, where the recursive structure of the tree translates to a recursion for said polynomial expression.  

Hence, in this section, let $G=\mc T$ be a rooted tree, where we will always interpret the root node as the (unique) source vertex, and the leaves as targets. 

To each node $w\in V(\mc T)$, we assign a polynomial $p[w]$ from $\IZ[x]$ recursively as follows: If $w$ is a leaf, define $p[w](x)=x$. Otherwise, denote the children of $w$ by $w_1,...,w_r$ for $r\geq 1$. Then, define 
\begin{equation}\label{eq:treerec}
 p[w](x):=x\left(1-\prod_{i=1}^r \left(1-p[w_i](x)\right)\right).
\end{equation}
Observe that in the case where $w$ only has a single child $w_1$, this simplifies to $p[w](x)=x p[w_1](x)$. Furthermore, if $\mc T_*$ is the fringe subtree of $\mc T$ rooted at $w$ (i.e. the subtree consisting of $w$ together with all its descendants), we also write $p[\mc T_*]:=p[w]$. In particular, $p[\mc T]:=p[\rt]$.

\begin{prop}\label{prop:treepol}
 For the continuous-time cutting model on a rooted tree $\mc T$, we have 
 \begin{equation}\label{eq:treepol}
  \bP[\mf S_c(\mc T)\geq x]= p[\mc T](e^{-x})
 \end{equation}
 for all $x\geq0$. Equivalently, one can interpret $p[\mc T](q)$ for $q\in[0,1]$ as the probability that $\Ber(q)$-site percolation on $\mc T$ contains a path from the root to a leaf. 
\end{prop}
\begin{proof}
 For $\mf S_c(\mc T)$ to be larger than $x$, there must exist a path  from the root node to a leaf which is still intact at time $x$. Now, use induction over the height of $\mc T$:
 
 If $\mc T$ consists of a single node, then $\bP[\mf S_c(\mc T)\geq x]=e^{-x}=p[\mc T](e^{-x})$. Now assume that \eqref{eq:treepol} holds true for all trees up to a certain height $h\geq 1$, and let $\mc T$ be a rooted tree of height $h+1$. Denote the children of the root in $\mc T$ by $w_1,...,w_r$ for some $r\geq 1$. Removing the root node creates $r$ trees $\mc T_1,...,\mc T_r$, which we endow with the new root nodes $w_1,...,w_r$, respectively. As the height of $\mc T_1,...,\mc T_r$ is at most $h$, applying the induction hypothesis yields
 \[
  \bP[\mf S_c(\mc T_i)\geq x]=p[\mc T_i](e^{-x}),\qquad \text{for }i=1,...,r.
 \]
 The existence of a desired path in $\mc T$ is equivalent to the root node not having been cut yet intersected with the event that not all of the subtrees $\mc T_i$ have been separated yet. Hence, 
 \begin{multline*}
  \bP[\mf S_c(\mc T)\geq x] 
  = e^{-x}\left(1-\prod_{i=1}^r \bP[\mf S_c(\mc T_i)< x]\right)\\
  = e^{-x}\left(1-\prod_{i=1}^r \left(1-p[\mc T_i](e^{-x})\right)\right)
  = p[\mc T](e^{-x})
 \end{multline*}
 by \eqref{eq:treerec}; thus concluding the proof.
\end{proof}

A \emph{transversal} in a rooted tree $\mc T$ is defined to be a subset of vertices that intersect every path from the root to a leaf. It then follows from the proof of Proposition \ref{prop:treepol} that $1-p[\mc T](1-q)$ yields the probability that a random set of vertices, containing each vertex independently with probability $q$, is a transversal of $\mc T$. It is this expression that was investigated in \cite{devroye2011note,2013transversals}. 

\begin{rem}
 There are some straightforward conclusions to be drawn from the previous proposition: The polynomials $p[\mc T](x)$ map the interval $[0,1]$ to itself, with $p[\mc T](0)=0$ and $p[\mc T](1)=1$. Furthermore, they are monotonically increasing on this interval, and by \eqref{eq:treerec}, $p[\mc T](x)\leq x$ for $x\in[0,1]$. 
\end{rem}

\begin{defn}\label{defn:subtrees}
 Let $\mc T$ be a rooted tree. 
 \begin{enumerate}[(i)]
  \item A subtree $\mc T_*$ is called \emph{faithful} if it contains the root of $\mc T$ as its own root and if every leaf of $\mc T_*$ is a leaf of $\mc T$. We denote the set of all faithful subtrees of $\mc T$ by $\msc F(\mc T)$, and the set of all faithful subtrees on $n$ vertices by $\msc F_n(\mc T)$. 
  \item Equivalently, a faithful subtree $\mc T_*\ssq \mc T$ can be seen as choosing a number of paths from the root to the leaves of $\mc T$. We hence denote the set of all faithful subtrees of $\mc T$ with exactly $\ell$ leaves by $\msc P_\ell(\mc T)$.
  \item A subtree $\mc T_*$ is called \emph{trimmed} if it contains the root of $\mc T$ as its own root and if $\deg_{\mc T_*}(v)=\deg_{\mc T}(v)$ for all non-leaves $v$ of $\mc T_*$.
 \end{enumerate}
\end{defn}

\begin{figure}[h]
 \includegraphics[scale=0.65]{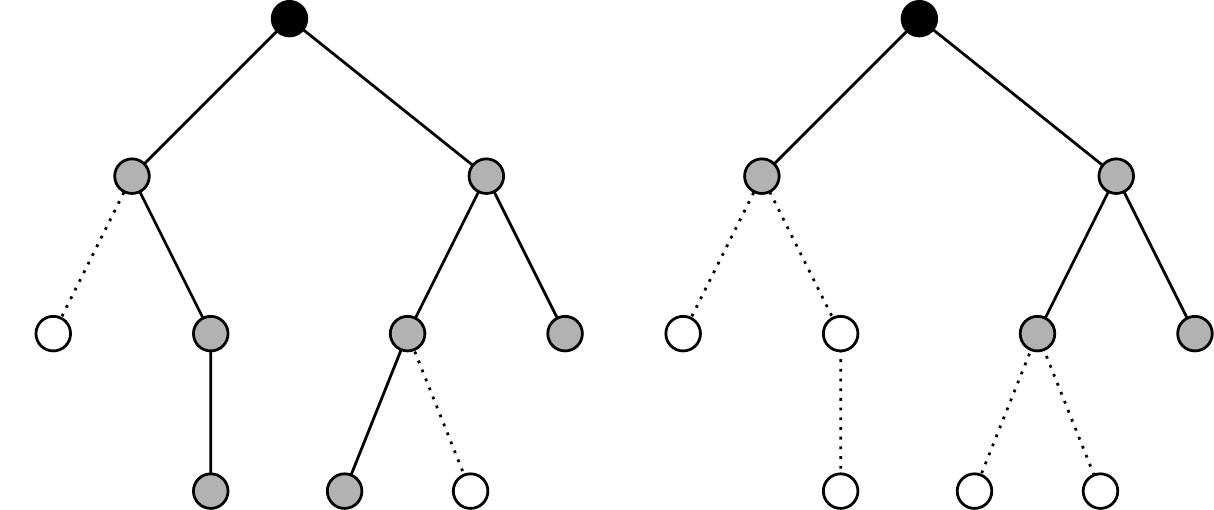}
 \caption{A faithful (left) and a trimmed (right) subtree of the same underlying rooted tree $\mc T$. The root is shown in black, with vertices belong to the subtrees being coloured grey. Dotted edges and white vertices belong to $\mc T$, but not to the subgraph.}
\end{figure}

The following propositions gives an alternative, combinatorial characterisation of $p[\mc T]$ which might be of independent interest:
\begin{prop}\label{prop:coeffs}
 Write $p[\mc T](x)=\sum_{j=0}^n a_j x^j$ with $a_j\in\IZ$ for $j=0,1,...,n$. Denote by $L(\mc T_*)$ the set of leaves of $\mc T_*$. Then,
 \begin{align}
  a_j &=\sum_{\mc T_*\in\msc F_j(\mc T)}  (-1)^{|L(\mc T_*)|+1}
        \label{eq:coeffs}\\
      &=\sum_{\ell\geq1} (-1)^{\ell+1} \left|\msc F_j(\mc T)\cap \msc P_\ell(\mc T)\right|. \label{eq:coeffs2}
 \end{align}
 In particular, $\deg(p[\mc T])=|\mc T|$ and $a_j=0$ for $j\leq \dist(\rt,L(\mc T))$.
\end{prop}
\begin{proof}
 First note that equation $\eqref{eq:coeffs}$ implies the other statements: Every subtree of $\mc T$ contains at most $|\mc T|$ vertices, and the only subtree with exactly $|\mc T|$ vertices is $\mc T$ itself. Hence $a_j=0$ for $j>|\mc T|$ and $a_{|\mc T|}=\pm 1$. On the other hand, any faithful subtree must contain at least one path from the root to $L(\mc T)$, and such a path requires at least $\dist(\rt,L(\mc T))+1$ vertices. 
 
 To show \eqref{eq:coeffs}, we employ $p[\mc T](e^{-x})=\bP[\mf S_c(\mc T)\geq x]$ from Proposition \ref{prop:treepol} and express the probability on the right-hand side in a different manner. To this end, write $\mc T_x$ for the tree obtained after time $x$. As observed already, $\{\mf S_c(\mc T)\geq x\}$ is the same as asking that $\mc T_x$ contains at least one path from the root to the leaves. Therefore, applying the inclusion-exclusion principle yields
 \begin{equation*} 
  \bP[\mf S_c(\mc T)\geq x]
  = \bP\left[\bigcup_{\mc T_*\in \msc P_1(\mc T)} \{\mc T_*\ssq \mc T_x\} \right]
  = \sum_{\ell\geq 1} (-1)^{\ell+1} \sum_{\mc T_*\in\msc P_\ell(\mc T)} \bP[\mc T_*\ssq \mc T_x].
 \end{equation*}
 Moreover, $\bP[\mc T_*\ssq \mc T_x]=(e^{-x})^{|V(\mc T_*)|}$, hence
 \begin{equation*}
  p[\mc T](e^{-x})=\bP[\mf S_c(\mc T)\geq x]
  =\sum_{\ell\geq 1} (-1)^{\ell+1} \sum_{\mc T_*\in\msc P_\ell(\mc T)} (e^{-x})^{|V(\mc T_*)|}.
 \end{equation*}
 Substituting $u$ for $e^{-x}$ and rearranging the sums on the right-hand side, we obtain the desired expressions for the coefficients of $p[\mc T]$. 
\end{proof}

\begin{cor}\label{cor:rootsep}
 We have
 \begin{equation}\label{eq:rootsep}
  \bP[G_{\mf S}=\emptyset]
  =\int_0^1 \frac{p[\mc T](u)}{u} \Di u.
 \end{equation}
 Moreover, we have for any subtree $\mc T_*\ssq \mc T$ containing the root node but none of the leaves:
 \begin{equation}\label{eq:tsep}
  \bP[G_{\mf S}=\mc T_*] = \sum_{v_*\in\partial \mc T_*} \int_0^1 u^{|\mc T_*|-1}(1-u)^{|\partial \mc T_*|-1} p[v_*](u) \Di u.
 \end{equation}
\end{cor}
\begin{proof}
 Both statements follow immediately from Propositions \ref{prop:sepgr} and \ref{prop:treepol}.
\end{proof}

\begin{lemma}\label{lem:treemon}
 Let $\mc T$ be a rooted tree. 
 \begin{enumerate}[(i)]
  \item If $\mc T_*\ssq \mc T$ is a faithful subtree then $p[\mc T_*](x)\leq p[\mc T](x)$ for all $x\in[0,1]$.
  \item If $\mc T_*\ssq \mc T$ is a trimmed subtree then $p[\mc T_*](x)\geq p[\mc T](x)$ for all $x\in[0,1]$.
 \end{enumerate}
\end{lemma}

\begin{proof}
 Assume $\mc T_*\ssq \mc T$ is faithful. If, at time $-\ln x$, there exists an uncut path in $\mc T_*$ connecting the root to $L(\mc T_*)$ then this same path also exists in $\mc T$ and connects the root to $L(\mc T)$. Hence we obtain from Proposition \ref{prop:treepol}:
 \begin{align*}
  p[\mc T_*](x) = \bP[\mf S_c(\mc T_*)\geq -\ln x] \leq \bP[\mf S_c(\mc T)\geq -\ln x] =p[\mc T](x).
 \end{align*}
 Assume now that $\mc T_*\ssq \mc T$ is trimmed. If there is a path in $\mc T$ connecting the root to $L(\mc T)$ at time $-\ln x$, then this path also connects the root to $L(\mc T_*)$. Thus, 
 \begin{align*}
  p[\mc T_*](x) = \bP[\mf S_c(\mc T_*)\geq -\ln x] \geq \bP[\mf S_c(\mc T)\geq -\ln x] =p[\mc T](x),
 \end{align*}
 analogously to the first case.
\end{proof}

We conclude this section by looking at a version of the tree polynomial for two variables that turns out to be useful for modifying trees at a fringe node: For a fixed leaf $v_*$, initialise $p_{v_*}[v_*](x,y)=y$ and all other leaves by $x$ as before. Then, recursively define $p_{v_*}[w]$ as in \eqref{eq:treerec}, replacing all terms $p[\cdot](x)$ by $p_{v_*}[\cdot](x,y)$. While in the original construction, we introduced one variable $x$ for every vertex, we still introduce one variable $x$ for every vertex bar $v_*$, which gets a $y$-variable. If the choice of $v_*$ is fixed or irrelevant for the purpose of the argument, we will drop the index and simply write $p[w](x,y)$. Analogously to before, we extend the notation to allow polynomials to be assigned to fringe subtrees. 

\begin{lemma}\label{lem:bivariate}
 In the setting just introduced, $p[\mc T](x,y)=a_1(x)y+a_2(x)$ for $a_1,a_2\in\IZ[x]$. Moreover, the map $A:C[0,1]\to C[0,1]$ sending $f$ to $a_1(x)f(x) + a_2(x)$ is a (strict) contraction with respect to the maximum norm, unless $\mc T$ is a path with the root positioned on an end vertex. 
\end{lemma}
\begin{proof}
 From the construction of the polynomial, $p[\mc T](x,x)=p[\mc T](x)$. Consider the maximal faithful subtree of $\mc T$ not containing the leaf $v_*$ to which $y$ was assigned, and denote it by $\mc T_{\setminus y}$. It follows from Proposition \ref{prop:coeffs} that $p[\mc T_{\setminus y}](x)$ is the part of $p[\mc T](x)$ that accounts for all those $y$-avoiding subtrees\footnote{We set $p[\emptyset](x)=0$ for the empty tree, to cover the case where $\mc T$ only had one leaf to begin with.}, thus $a_2(x)=p[\mc T_{\setminus y}](x)$. Consequently, $p[\mc T]-p[\mc T_{\setminus y}]$ is the polynomial corresponding to all faithful subtrees containing the $y$-leaf. However, we still need to exchange one $x$-variable by a $y$-variable, thus $a_1(x)=\frac{1}{x}\left(p[\mc T](x)-p[\mc T_{\setminus y}](x)\right)$. Observe that this is indeed a polynomial, as every one-variate tree polynomial is divisible by $x$.
 
 To show that the operator $A$ is a contraction, it suffices to verify that $\|a_1\|_{C[0,1]} <1$, since 
 \begin{equation}\label{eq:opnorm}
  \|Af-Ag\|_{C[0,1]} 
  = \|a_1(f-g)\|_{C[0,1]} 
  \leq \|a_1\|_{C[0,1]}\cdot\|f-g\|_{C[0,1]},
 \end{equation}
 as desired. 
 To this end, recall that by Proposition~\ref{prop:treepol}, the polynomial $p[\mc T_{\setminus y}](x)$ can be interpreted as the probability that $\mc T_{\setminus y}$ contains any path from the root to a leaf of $\mc T$ different from $v_*$, when nodes are deleted independently with probability $1-x\in[0,1]$, and kept otherwise. Since the analogous statement is also true for $p[\mc T](x)$, this gives the following probabilistic interpretation of $a_1(x)=\frac{1}{x}\left(p[\mc T](x)-p[\mc T_{\setminus y}](x)\right)$: The event that the only path from the root to a leaf in $\mc T$ is the path to $v_*$, conditioned on the node $v_*$ being present, has probability $a_1(x)$. From this, it follows immediately that $0\leq a_1(x) < 1$ for $x\in[0,1]$ unless $v_*$ is the unique leaf of $\mc T$. The inequality $\|a_1\|_{C[0,1]}<1$ now follows from the fact that $a_1$ is a polynomial and therefore continuous. 
\end{proof}

Let $(\msc T_n)_{n\in\IN}$ be a family of finite sets of rooted trees, say \[
 \msc T_n = \left\{\mc T_{n,1},\dots,\mc T_{n,k_n} \right\}
\]
for integers $k_1,k_2,\dots$. Given any rooted tree $\mc T$ and any positive integer $n$, we denote by $M_n(\mc T)$ the rooted tree which, upon removal of its root node, decomposes into one copy of $\mc T, \mc T_{n,1},\dots,\mc T_{n,k_n}$ each, with the roots of these trees being the children of the root in $M_n(\mc T)$. We can think of $M_n(\,\cdot\,)$ for fixed $n$ as being a map from the set of rooted trees to itself. Accordingly, those maps can be composed, and we use the shorthand notation 
\[
 M_{n_\ell,\dots,n_1}(\,\cdot\,)=M_{n_\ell}\circ \dots\circ M_{n_1}(\,\cdot\,)
\] 
for any finite sequence $n_1,\dots,n_\ell$ of positive integers.

\begin{thm}\label{thm:fringe}
 Let $\bullet$ be the rooted tree consisting of one vertex and let $\mc T_*$ be a fixed rooted tree.
 Consider a family $(\msc T_n)_{n\in\IN}$ of finite sets of rooted trees, as above. If there is a constant $c<1$ such that 
 \begin{equation}\label{eq:fringecond}
  \Bigg\|x\prod_{j=1}^{k_n} \big(1-p[\mc T_{n,j}](x)\big)\Bigg\|_{C[0,1]}\leq c \ \text{ for all $n\in\IN$,}
 \end{equation}
 then 
 \[
  \big\| p\left[M_{n,\dots,1}(\mc T_*)\right](x)- 
  p\left[M_{n,\dots,1}(\bullet)\right](x) \big\|_{C[0,1]} \to 0 
 \]
 as $n\to\infty$. 
\end{thm}
\begin{proof}
 Let $m\in \IN$, and let $\mc T$ be any rooted tree. By virtue of \eqref{eq:treerec}, we have 
 \[
  p[M_m(\mc T)](x)
  =x-x\prod_{j=1}^{k_m}\big(1-p[\mc T_{m,j}](x)\big)
   + xp[\mc T](x) \prod_{j=1}^{k_m}\big(1-p[\mc T_{m,j}](x)\big).
 \]
 Simultaneously, fixing the leaf $\bullet$ in $M_k(\bullet)$ and assigning to it the variable $y$, we obtain the bivariate polynomial 
 \[
  p_{\bullet}[M_m(\bullet)](x,y)
  =x-x\prod_{j=1}^{k_m}\big(1-p[\mc T_{m,j}](x)\big) + xy\prod_{j=1}^{k_m}\big(1-p[\mc T_{m,j}](x)\big).
 \]
 In particular,
 \[
  p[M_m(\mc T)](x)=p_\bullet[M_m(\bullet)](x,p[\mc T](x)).
 \]
 Hence the operator $A_m:p[\mc T](x)\mapsto p[M_m(\mc T)](x)$ has the form described in Lemma~\ref{lem:bivariate} and is therefore a strict contraction with an operator norm bounded by 
 \[
  \Bigg\|x\prod_{j=1}^{k_m} \big(1-p[\mc T_{m,j}](x)\big)\Bigg\|_{C[0,1]},
 \]
 according to the estimate~\eqref{eq:opnorm}. Therefore, iterating \eqref{eq:opnorm} yields
 \begin{align*}
  &\big\| p\left[M_{n,\dots,1}(\mc T_*)\right](x)- p\left[M_{n,\dots,1}(\bullet)\right](x) \big\|_{C[0,1]}\\
  &\qquad\leq \prod_{i=1}^n \Bigg\|x\prod_{j=1}^{k_i} \big(1-p[\mc T_{i,j}](x)\big)\Bigg\|_{C[0,1]} \cdot\big\|p[\mc T_*](x)-p[\bullet](x)\big\|_{C[0,1]}\\
  &\qquad\leq c^n \big\|p[\mc T_*](x)-p[\bullet](x)\big\|_{C[0,1]} \to 0
 \end{align*}
 as $n\to\infty$. 
\end{proof}

The general message of Theorem~\ref{thm:fringe} is thus that two trees that only differ in a fringe tree which is rooted far away from the root node will have approximately the same polynomial function $p[\,\cdot\,]$ over $[0,1]$, provided condition~\eqref{eq:fringecond} holds. We will apply this idea in Proposition~\ref{prop:binary} below.

\section{Examples: Caterpillar trees, stars, and complete binary trees}

\begin{figure}[!h]
 \includegraphics[scale=0.65]{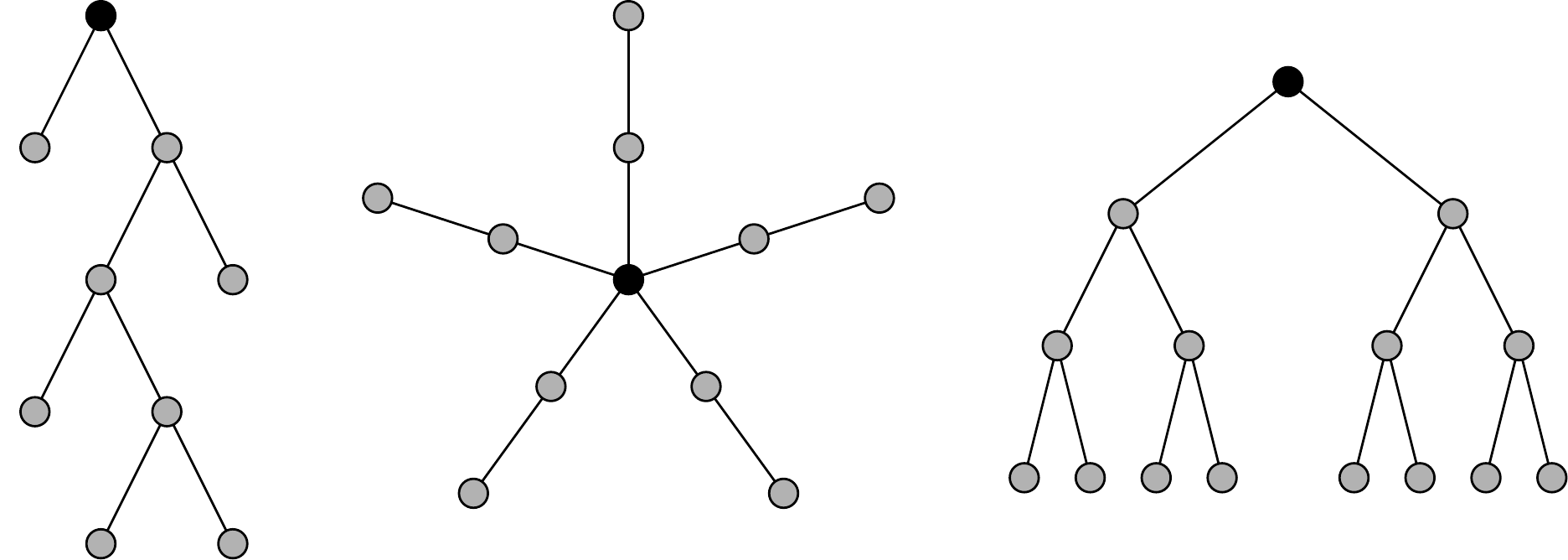}
 \caption{The trees $\cater_4$, $\Star_{3,2}$ and $\binary_4$ from left to right, with the root shown in black.}
\end{figure}

\begin{ex}\label{ex:caterpillar}
 Let $\cater_n$ denote the binary caterpillar tree on $2n+1$ nodes, that is, a rooted tree where every node is either a leaf, or has at least one child which is a leaf. The associated polynomials then satisfy the recurrence relation 
 \begin{align*}
  p[\cater_{n+1}](x)
  &=x(1-(1-x)(1-p[\cater_n](x)))\\
  &=x^2+xp[\cater_n](x)-x^2p[\cater_n](x)
 \end{align*}
 with $p[\cater_0](x)=x$. By the Banach fixed point theorem in $C[0,1]$, this recurrence converges against the unique solution of the associated fixed point equation $f=x^2+xf-x^2f$, which is $f(x)=\frac{x^2}{x^2-x+1}$. Finally, since all functions involved are bounded on the interval $[0,1]$, we obtain 
 \[
  \bP[\mf S(\cater_n)=\mf C(\cater_n)] \longrightarrow \int_0^1 \frac{x}{1-x+x^2}\,\Di x = \frac{\pi}{3\sqrt{3}} \approx 0.6046
 \]
 as $n\to \infty$ by applying the dominated convergence theorem to Corollary \ref{cor:rootsep}. 
 
 Moreover, we can use these very same tools to compute the limiting distribution of $|\cater_{n,\mf S}|$. Indeed, every admissible graph must be a path of some length, say $m\geq0$, which then has $m$ leaves as boundary plus one further node that induces a subtree isomorphic to $\cater_{n-m}$. Hence by Corollary \ref{cor:rootsep}, we have for all $m\geq0$ as $n\to\infty$:
 \begin{align}
  \bP[|\cater_{n,\mf S}|=m] 
  &=\int_0^1 x^{m-1}(1-x)^m (mx+p[\cater_{n-m}](x))\Di x\notag \\
  &\longrightarrow m \Beta(m+1,m+1) + \int_0^1 x^m(1-x)^m \frac{x}{1-x+x^2} \Di x.\label{eq:caterlim1} 
 \end{align}
 For evaluating this limit, we note first that
 \begin{align*}
  \frac{x^{m+1}(1-x)^m}{x^2-x+1}
  &=\frac{x}{x^2-x+1}-\sum_{j=0}^{m-1} x^{j+1}(1-x)^j
 \end{align*}
 so that the integral in \eqref{eq:caterlim1} equals
 \begin{align*}
  \int_0^1 \frac{x^{m+1}(1-x)^m}{1-x+x^2} \Di x 
  &=\int_0^1 \frac{x}{x^2-x+1} \Di x - \sum_{j=0}^{m-1} \Beta(j+2,j+1) \\
  &= \frac{\pi}{3\sqrt{3}} - \sum_{j=0}^{m-1} \frac{1}{j+1} \binom{2j+2}{j+1}^{-1}.
 \end{align*}
 Using these results, we obtain for the limit in \eqref{eq:caterlim1}:
 \begin{align}\label{eq:caterlim}
  \lim_{n\to\infty} \bP[|\cater_{n,\mf S}|=m]
  &= \int_0^1 \left(m(x-x^2)^m + \frac{x(x-x^2)^m}{1-x+x^2}\right) \Di x\\
  &= \frac{\pi}{3\sqrt{3}} + \frac{m}{(m+1)\binom{2m+1}{m}} - \sum_{j=0}^{m-1} \frac{1}{(j+1)\binom{2j+2}{j+1}}.
 \end{align}
 These limiting expressions form indeed a probability distribution on the non-negative integers, as can be comfortably shown with Theorem \ref{thm:tight}: In the notation of this theorem, we have $|\partial_S A|=m+1$ for all admissible graphs $A$ on $m$ vertices, thus $b=1$, and the radius of convergence of 
\[
 \sum_{m=0}^\infty a_mx^m=\sum_{m=0}^\infty (m+1)x^m = \frac{1}{(1-x)^2}
\]
 equals $1>1/4$. Hence, the random variables $|\cater_{n,\mf S}|$ form a tight, vaguely converging sequence, which must thus also converge in distribution. 
\end{ex}

\begin{rem}\label{rem:catermom}
 It is also possible to directly verify that \eqref{eq:caterlim} defines a random variable. 
 To see this, observe that the integrand in \eqref{eq:caterlim} is non-negative. We thus obtain from the monotone convergence theorem that
 \begin{align}\label{eq:cater1}
  \sum_{m=0}^\infty \lim_{n\to\infty} &\bP[|\cater_{n,\mf S}|=m]\notag \\
  &= \int_0^1 \left( \sum_{m=0}^\infty m(x-x^2)^m + \frac{x}{1-x+x^2} \sum_{m=0}^\infty (x-x^2)^m \right) \Di x\notag \\
  &= \int_0^1 \frac{2x-x^2}{(1-x+x^2)^2} \Di x =1,
 \end{align}
 as required.
 
 Furthermore, the estimate 
 \[
  m(x-x^2)^m + \frac{x(x-x^2)^m}{1-x+x^2} \leq \frac{m+1}{4^m}, \quad x\in[0,1],
 \]
 implies that a random variable $X$ having the probability distribution defined by \eqref{eq:caterlim} has exponential tails. In particular, all its moments exist, and they can be computed in a similar fashion to \eqref{eq:cater1}. Omitting the details, we obtain 
 \begin{align*}
  \bE[X] &= \frac{1}{54}(9+4\sqrt{3}\pi) \approx 0.5697\\
  \Var[X] &= \frac{1}{972}(783 + 16\sqrt{3} \pi - 16 \pi^2) \approx 0.7327.
 \end{align*}

\end{rem}

\begin{ex}\label{ex:star}
 Let $\Star_{r,n}$ denote the rooted tree on $rn+1$ vertices which, upon removal of the root node, splits into $r$ copies of a path on $n$ vertices. Hence, the tree polynomial of $\Star_{r,n}$ is given by 
 \begin{equation}\label{eq:starpol}
  p[\Star_{r,n}](x)=x(1-(1-x^n)^r)
 \end{equation}
 and the probability of separating the tree by cutting the root node is therefore given by 
 \begin{multline}\label{eq:starprob}
  \bP[\mf S(\Star_{r,n})=\mf C(\Star_{r,n})] 
  = \int_0^1 1-(1-x^n)^r\Di x\\
  = 1-\frac{1}{n}\int_0^1 y^{\frac{1}{n}-1}(1-y)^r \Di y 
  = 1-\frac{1}{n}\Beta\left(\frac{1}{n},r+1\right)
 \end{multline}
 after substituting $y=x^n$, according to Corollary~\ref{cor:rootsep}. Letting the size of the tree grow to infinity, we obtain different asymptotic behaviours of this probability, depending on the relation between $r$ and $n$. To see this, let $\rho\in(0,1)$ and set $r=r(n)=(1-\rho)^{-n}-1$. In the interest of brevity, write $z:=(1-\rho)^{-n}=r(n)+1$. Then, using Stirling's approximation for the Gamma function, we obtain
 \begin{multline*}
  \frac{1}{n}\Beta\left(\frac{1}{n},r(n)+1\right) 
  = \gC\left(1+\frac{1}{n}\right)
    \frac{\gC(z)}{\gC\left(z-\frac{\ln(1-\rho)}{\ln z}\right)}\\
  \sim \sqrt{\frac{z-\frac{\ln(1-\rho)}{\ln z}}{z}} \cdot 
  \frac{e^{z-\frac{\ln(1-\rho)}{\ln z}}}{e^z} \cdot
  \frac{z^z}{\left(z-\frac{\ln(1-\rho)}{\ln z}\right)^{z-\frac{\ln(1-\rho)}{\ln z}}} 
  \longrightarrow 1-\rho
 \end{multline*}
 as $n$ (and therefore $z$) tends to infinity. Hence, in light of \eqref{eq:starprob}, and by employing monotonicity properties w.r.t. $r$ and $n$, we get the following results:
 \begin{enumerate}[(a)]
  \item If the number $r$ of rays grows subexponentially with $n$ then we have $\bP[\mf S(\Star_{r,n})=\mf C(\Star_{r,n})] \to 0$. In particular, this covers the intuitively obvious case when the number or rays is bounded. 
  \item If $r\sim (1-\rho)^{-n}$ then by the above $\bP[\mf S(\Star_{r,n})=\mf C(\Star_{r,n})] \to \rho$.
  \item If $r$ grows superexponentially with $n$ (or, if $n$ is even bounded) then $\bP[\mf S(\Star_{r,n})=\mf C(\Star_{r,n})] \to 1$.
 \end{enumerate}
\end{ex}

\begin{ex}\label{ex:binary}
 Denote by $\binary_n$ the full complete rooted binary tree on $2^n-1$ vertices, this being the rooted binary tree having $2^h$ vertices at every height $h=0,...,n-1$. Observe that thus $\binary_1$ is the tree consisting of only the root node, and that $\binary_{n+1}$ splits into two copies of $\binary_n$ upon removing the root node. Thus, the associated polynomials satisfy the recurrence equation 
 \begin{align*}
  p[\binary_{n+1}](x) &= x(1-(1-p[\binary_n](x))^2)\\
    &= 2xp[\binary_n](x) - x p[\binary_n](x)^2
 \end{align*}
 with $p[\binary_1](x)=x$. The fixed-point equation $f=2xf-xf^2$ admits two solutions, $f\equiv 0$ and $f(x)=2-\frac{1}{x}$. For convenience, let $\gp(x):= \max\left\{0,2-\frac{1}{x}\right\}$. Direct verification reveals that if a function $g$ satisfies $g(x) \geq \gp(x)$ for all $x$, then $2xg(x)-xg(x)^2\geq\gp(x)$ as well. Furthermore, the sequence of polynomials $p[\binary_n]$ decreases monotonically pointwise by Lemma~\ref{lem:treemon}, and hence converges pointwise against $\gp(x)$. 
 
 In light of Proposition \ref{prop:perc}, this result should not be surprising: The sequence of rooted trees $\binary_n$ satisfies all the conditions, and the function $\gp(x)$ indeed equals the probability that the root node is contained in an infinite cluster of $\Ber(x)$-site percolation, as can be verified independently.\footnote{For example, in \cite{grimmett1999percolation}, p. 256, a similar recursive argument is used to determine the probability of the root being contained in an infinite cluster of $\Ber(p)$-bond percolation to be $\gp(p)/p$. While not the same, there is a bijection between edges and non-root vertices in a rooted tree by mapping any edge to the endpoint further away from the root. This allows us to translate between site and bond percolation and explains Grimmett's additional factor of $p^{-1}$.} 
 
 Continuing with our analysis, the probability of the remaining tree at separation being empty now follows handily from Corollary~\ref{cor:rootsep}: 
 \[
  \lim_{n\to\infty}\bP[|\binary_{n,\mf S}|=0] = \int_0^1 \frac{\gp(x)}{x}\Di x 
  = \ln 4 - 1\approx 0.3863
 \]
 
 In a similar fashion, we can continue to determine the limiting probability of separation graphs of any size $m\geq 0$: Since there are $C_m = \frac{1}{m+1}\binom{2m}{m}$-many\footnote{These are, of course, the Catalan numbers, \cite[A000108]{OEIS}} subtrees of the infinite rooted binary tree on $m$ vertices, and each of those has $m+1$ boundary vertices, we claim that 
 \begin{equation}\label{eq:binarylimit1}
  \lim_{n\to\infty}\bP[|\binary_{n,\mf S}|=m] = \binom{2m}{m}\int_0^1 x^{m-1}(1-x)^m \gp(x) \Di x.
 \end{equation}
 This can be derived in a manner analogous to equation \eqref{eq:caterlim1}, where the change between limit and integral can be motivated by recalling the estimate $p[\binary_n](x)\leq x$, so that the integrand is bounded by $x^m(1-x)^m$. Hence, by the dominated convergence theorem, we arrive at \eqref{eq:binarylimit1}.
 
 Observe that, in the notation of Theorem \ref{thm:tight}, the sequence $a_m=\binom{2m}{m}$ has generating function $\frac{1}{\sqrt{1-4x}}$ (cf. \cite[A000984]{OEIS}), thus violating the integrability condition \eqref{eq:integrable} for the proper constant $b=1$. However, we can check by hand that \eqref{eq:binarylimit1} defines a probability distribution with the following computation, again relying on the generating function of $\binom{2m}{m}$:
 \begin{align}\label{eq:binarytight}
  \sum_{m=0}^{\infty} \binom{2m}{m} &\int_0^1 x^{m-1}(1-x)^m \gp(x)\Di x\notag\\
  &= \int_{1/2}^1 \left(\sum_{m=0}^\infty \binom{2m}{m}(x-x^2)^m\right) \cdot \frac{2x-1}{x^2}\Di x\notag\\
  &= \int_{1/2}^1 \frac{1}{\sqrt{1-4x+4x^2}}\cdot \frac{2x-1}{x^2}\Di x\notag\\
  &= \int_{1/2}^1 \frac{1}{x^2} \Di x = 1.
 \end{align}
 Note also that a random variable $X$ having the probability distribution defined by \eqref{eq:binarylimit1} does not have a finite first moment: Imitating the approach of \eqref{eq:binarytight} leads to 
 \[
  \bE[X] = \sum_{m=0}^\infty m\binom{2m}{m} \int_0^1 x^{m-1}(1-x)^m\gp(x)\Di x
  = \int_{1/2}^1 \frac{2(1-x)}{x(2x-1)^2}\Di x
 \]
 where the integral on the right-hand side diverges. 
\end{ex}

Finally, we can use Lemmas \ref{lem:treemon} and \ref{lem:bivariate} to verify that the same limiting distribution also holds if we consider the sequence of complete binary trees on $n$ vertices (of which the full complete binary trees are merely a subsequence):  
 
Denote by $\mc T_n$ the complete binary tree on $n$ vertices, this being the binary tree having $2^k$ vertices at height $k$ for $0\leq k< \lfloor \lg n\rfloor =: m$, with the remaining $n-2^m+1$ vertices at height $m$ in their left-most positions.

\begin{prop}\label{prop:binary}
 With $\gp(x)=\max\left\{0,2-\frac{1}{x}\right\}$ as above, we have $p[\mc T_n](x) \to \gp(x)$ uniformly over $[0,1]$, as $n\to\infty$. 
 
 Consequently, the limiting distribution of $\left|\mc T_{n,\mf S}\right|$ coincides with that of $\left|\binary_{n,\mf S}\right|$ and is given by equation~\eqref{eq:binarylimit1}.
\end{prop}

\begin{proof}
 If $n$ is odd, then the number of vertices at height $m+1$ is even, so we have $\binary_{m-1} \ssq \mc T_n\ssq \binary_m$, where in both inclusions, the smaller tree is a trimmed subtree of the larger according to Definition~\ref{defn:subtrees}. Accordingly, we obtain from Lemma~\ref{lem:treemon} that 
 \[
  p[\binary_{m-1}](x)\geq p[\mc T_n](x)\geq p[\binary_m](x)
 \]
 for all $x\in[0,1]$. Hence $p[\mc T_n](x)\to \gp(x)$ pointwise (and uniformly) for odd $n\to \infty$. 
 
 For even $n$, observe that $\mc T_n$ and $\mc T_{n-1}$ only differ in a fringe subtree of height 2. Indeed, let $\text{root},v_{m-1},\dots,v_1$ be the path from the root to the parent of the unique leaf $\ell$ without a sibling vertex in $\mc T_n$. Each of the vertices $v_i$ is also in $\mc T_{n-1}$, and in both trees, they have a unique sibling, $w_i$, for $1\leq i< m$ (see Figure~\ref{fig:binary}). Writing $\mc T_n^{(v)}$ for the fringe subtree of $\mc T_n$ rooted at a vertex $v$, and setting $\msc T_j:= \left\{\mc T_n^{(w_j)}\right\}$, we observe that condition~\eqref{eq:fringecond} is satisfied: Since every $\mc T_n^{(w_j)}$ is a non-empty full complete binary tree, we have 
 \[
  \gp(x)\leq p\left[\mc T_n^{(w_j)}\right](x)\leq x
 \]
 for $x\in[0,1]$ by the results of Example~\ref{ex:binary}. Accordingly, 
 \[
  \left\|x\left(1-p\left[\mc T_n^{(w_j)}\right](x)\right)\right\|_{C[0,1]}
  \leq \left\| x(1-\gp(x)) \right\|_{C[0,1]}=\frac{1}{2}
 \]
 for all $j$. Therefore, 
 \[
  \big\| p[\mc T_n]-p[\mc T_{n-1}]\big\|_{C[0,1]} \to 0
 \]
 as $m\to \infty$ for even $n$ by Theorem~\ref{thm:fringe}, yielding the desired convergence. (In fact, repeating the estimates from the proof of Theorem~\ref{thm:fringe} reveals the more concise estimate
 $\big\|p[\mc T_n]-p[\mc T_{n-1}]\big\|\leq 2^{1-m}$).
 
 \begin{figure}[h]
  \includegraphics[]{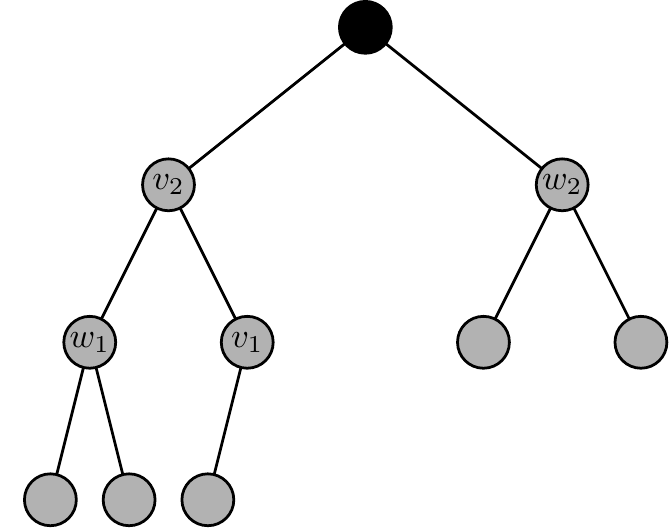}
  \caption{The vertices $v_i$ and $w_i$ from the proof of Proposition~\ref{prop:binary} in $\mc T_{10}$.}\label{fig:binary}
 \end{figure}

 The second claim concerning the limiting distribution of the size of the separation graph follows in the same way as when it was first derived in Example~\ref{ex:binary}.
\end{proof}

Observe that by this proposition and by \eqref{eq:binarytight}, the random variables $\left|\mc T_{n,\mf S}\right|$ converge in distribution and are therefore tight. By Remark~\ref{rem:tight} this means that the limit law obtained by Janson in \cite[Theorem~1.1]{janson2004random} for $\mf C(\mc T_n)$ holds also for $\mf S(\mc T_n)$. More explicitly, if we denote by $\{x\}:= x-\lfloor x\rfloor$ the fractional part of $x\in\IR$ then we obtain the following
\begin{cor}\label{cor:binary}
 Let $n\to\infty$ such that $\{\lg n - \lg\lg n\}\to \gc\in[0,1]$. Write $f(\gc):=2^\gc-1-\gc$. Let $W_{\gc}$ be a random variable with an infinitely divisible distribution with characteristic function 
 \[
  \bE\left[e^{itW_{\gc}}\right] 
  = \exp\left(if(\gc)t+\int_0^\infty \left(e^{itx}-1-itx\mathds{1}_{\{x<1\}}\right) \Di \nu_\gc(x)\right),
 \]
 with the Lévy measure $\nu_\gc$ being supported on $[0,\infty)$ and having density $\Di \nu_\gc = 2^{\{\lg x+\gc\}}x^{-2}\Di x$.
 Then
 \[
  \frac{\lg^2 n}{n}\left(\mf S(\mc T_n)-\frac{n}{\lg n}-\frac{n\lg\lg n}{\lg^2 n}\right) \dto - W_\gc
 \]
\end{cor}


\section{Further questions}

We use this final section to present several, deliberately broad questions or remarks that could lead to interesting future research.

\subsection*{1} Determine the asymptotic distributions of $\mf S(G^{(n)})$ and $\left|G^{(n)}_{\mf S}\right|$ for other families of deterministic and random trees, in continuation to the wide variety of work done on the cutting number of trees. The author hopes to answer this for conditioned Galton-Watson trees in a follow-up paper. 

\subsection*{2} What happens if the roles of $S$ and $T$ are exchanged? This promises to be non-trivial already for rooted trees and their leaves. Moreover, for which graphs and which choices of $S,T$ are the random variables $\mf S(G;S,T)$ and $\mf S(G;T,S)$ equal in distribution?

\subsection*{3} How to evaluate the asymptotic distribution of $\mf S$ directly, without relying on previous knowledge of $\mf C$ as in Corollary~\ref{cor:binary}?

\subsection*{4} On trees, the difference between edge-cutting and vertex-cutting is usually negligible because there is a bijection between edges and non-root vertices, assigning the endpoint further away from the root to each edge. For general graphs, no such bijection exists. However, it is easy to see that the edge-cutting process on a graph $G$ is exactly the vertex-cutting process on the line graph of $G$. This therefore raises the question: How is the separation time on $G$ related to the separation time on the line graph of $G$?

\subsection*{5} For which sequences of graphs $G^{(n)}$ exhausting a locally finite infinite $G$ (with fixed sources and targets) are the random variables $\left|G^{(n)}_{\mf S}\right|$ \emph{not} tight? In this case, what can be said about the structure of the remaining graph?

\subsection*{6} By definition, the separation number is the number of cuts required to separate two subsets $S,T\ssq V$ from each other. Starting in a graph with a high connectivity (say, by having $k\geq 2$ vertex-disjoint paths from $S$ to $T$ in $G$) we can ask for the number of cuts required to reduce the connectivity to some $j\leq k$. The case $j=0$ specializes to the separation number as we defined it. However, since the notion of boundary we used to prove the results in Section 4 is not relevant anymore if $j>0$, the question is what kind of statements can be obtained for the more general case.\\

\subsection*{Acknowledgements} The author wishes to thank his academic advisors, Cecilia Holmgren and Svante Janson, for their generous support and many helpful remarks and discussions. The gratitude is extended to Jonas Sjöstrand, who pointed out a mistake in Corollary~\ref{cor:complete} and suggested a simplification for the formulas there. This work was partially supported by grants from Knut and Alice Wallenberg Foundation, the Ragnar Söderberg Foundation, and the Swedish Research Council.


\nocite{*}

\bibliographystyle{alphaurl}
\bibliography{sep_bib}
 
\end{document}